\documentclass[12pt]{article}
\usepackage{amsmath,amsfonts,amssymb,amsthm}
\newcommand{\N}{\mathbb N}
\newcommand{\Z}{\mathbb Z}
\newcommand{\R}{\mathbb R}
\newcommand{\twosum}[2]{\sum_{\substack{#1\\#2}}}
\newcommand{\ca}{\mathcal A}
\newcommand{\p}{\mathcal P}
\newcommand{\e}{\mathcal E}

\newtheorem{defn}{Definition}[section]
\newtheorem{thm}[defn]{Theorem}
\newtheorem{lem}[defn]{Lemma}

\title{Diophantine Approximation with Products of Two Primes}
\author{A.J. Irving\\
Mathematical Institute, Oxford}
\date{}
\begin{document}
\maketitle

\section{Introduction}

Let $\|x\|$ denote the distance from the real number  $x$ to the
nearest integer.  Given an irrational $\alpha$, we are interested in
the problem of determining the values of $\tau$ for which there are
infinitely many prime solutions, $p$, to the Diophantine inequality  
\begin{equation}\label{diophantineineq}
\|p\alpha\|\leq p^{-\tau}.
\end{equation}
It is an easy consequence of the Generalised Riemann Hypothesis that
any $\tau<\frac{1}{3}$ is admissible.  This was proved unconditionally
by Matom{\"a}ki, \cite{mat}, and is currently the strongest result
known.  Progress on this problem began with Vinogradov,
\cite{vinogradov}, who proved that we can take any $\tau<\frac{1}{5}$.
Vaughan, \cite{vaughan}, simplified the proof whilst improving the
exponent to $\tau<\frac{1}{4}$.  In both of these works  an asymptotic
formula for the number of prime solutions is proved.  Harman,
\cite{har83}, introduced a sieve method to the problem.  This only
gives a lower bound for the number of solutions but this is
sufficient.  He increased the size of $\tau$ to $\tau<\frac{3}{10}$,
improving this  in \cite{har96} to $\tau<\frac{7}{22}$.  These results
of Harman used identical arithmetic information to the results of
Vaughan; the improvements were in the sieve method.  Heath-Brown and
Jia, \cite{rhbjia}, found new arithmetic information which they were
able to use to get $\tau<\frac{16}{49}$.  Matom{\"a}ki, by using
results on averages of Kloosterman sums, was able to extend this to
handle any $\tau<\frac{1}{3}$.   

If we only require the solutions  of (\ref{diophantineineq}) to have at most two prime factors
then the problem is considerably easier as classical sieve methods may
be used.  In particular Harman, \cite[Theorem 2]{har83}, states that
any $\tau<0.46$ is sufficient.   One reason for a stronger result is
that the parity problem of sieve theory is no longer an issue.  In
order to circumvent the parity problem and detect primes it is
necessary to prove estimates for bilinear forms, known as ``Type II''
sums.  Matom{\"a}ki, \cite{mat},  describes all the estimates known
for $\tau<\frac{1}{3}$ but none of her  proofs are valid for $\tau\geq
\frac{1}{3}$.  We  will prove a Type II bound  in which one may take
$\tau$ slightly larger than $\frac13$.  This estimate is too weak to
show the existence of prime solutions to (\ref{diophantineineq}).  It
does, however, show that there are solutions which have precisely two
prime factors.  Hence we can break the parity barrier for some
$\tau>\frac13$. 

We are also interested in the set $\p_3(b)$ of $3$-digit palindromes
in base $b$.  We say that a number is palindromic in base $b$ if its
digits in base $b$ are the same when reversed.  Thus  
$$\p_3(b)=\{j(b^2+1)+kb:j\in (0,b)\cap\Z,k\in [0,b)\cap\Z\}.$$
As we shall see in Section \ref{proofs}, elements in this set correspond closely to solutions of (\ref{diophantineineq}) when $\tau=\frac13$.  We may therefore also conclude that $\p_3(b)$ contains
numbers with precisely two prime factors provided that $b$ is
sufficiently large. 

To handle both of these problems simultaneously we work with the
following  set.  For a natural number $q$, positive reals $x,z$ and an
integer $a$ with $(a,q)=1$ we let  
$$\ca=\ca(x,q,z,a)=\{n\in (\frac{x}{4},x]:n\equiv ak\pmod q\text{ for some
}k\in [0,z)\cap\Z\}.$$ 
For a fixed constant $\tau\in (0,1)$ we shall only consider the case when 
$$z\in \left[\frac{1}{2}q^{\frac{1-\tau}{1+\tau}},2q^{\frac{1-\tau}{1+\tau}}\right]$$
 and 
$$x\in \left[\frac{1}{2}q^{\frac{2}{1+\tau}},2q^{\frac{2}{1+\tau}}\right].$$
All implied constants in our results may depend on $\tau$.  Observe that
$zq\asymp x$. 

Our aim is to estimate Type I and Type II sums for the set
$\ca$ and use them to prove the following. 

\begin{thm}\label{mainthm}
Suppose $\tau<\frac{8}{23}$ is fixed.  Let $\e_2$ be the set of natural numbers having precisely $2$ prime factors.  With the above definitions and hypotheses we have 
$$\#(\ca\cap \e_2)\gg \frac{z^2}{\log z},$$
provided that $q$ is sufficiently large in terms of $\tau$.
\end{thm}

A result of this form for $\tau<\frac13$ would follow immediately from Vaughan's work, \cite{vaughan}.  The key new idea  to handle larger $\tau$ is our Type II estimate, Theorem \ref{typeii}.

This theorem enables us to prove the following results regarding the problems discussed above.

\begin{thm}\label{diophantine}
Let $\alpha$ be irrational.  For any $\tau<\frac{8}{23}$ there exist infinitely many $n\in\e_2$ such that 
$$\|n\alpha\|\leq  n^{-\tau}.$$ 
\end{thm}

\begin{thm}\label{palindrome}
For all sufficiently large $b$ we have 
$$\#(\p_3(b)\cap\e_2)\gg \frac{b^2}{\log b}.$$
\end{thm}

\subsection*{Acknowledgements}

This work was completed as part of my DPhil, for which I was funded by EPSRC grant EP/P505666/1. I am very grateful to the
EPSRC for funding me and to my supervisor, Roger Heath-Brown, for all
his help and advice. 

\section{Notation and Useful Results}

We will write $e(x)=e^{2\pi ix}$ and 
$$\mathbf{1}_\ca(n)=\begin{cases}
1 & n\in \ca\\
0 & n\notin \ca.\\
\end{cases}$$
We will use the notation $n\sim N$ to mean $N<n\leq 2N$ and similarly $n\asymp N$ to mean $aN<n\leq bN$ for some $a,b>0$.  We will also need the Fourier transform $\hat f$ of the function $f$, defined by 
$$\hat f(x)=\int_{-\infty}^\infty f(t)e(-tx)\,dt.$$
We will write $\tau(n)$ for the number of divisors of $n$.  It is well known that for any $\epsilon>0$ we have $\tau(n)\ll_\epsilon n^\epsilon$.  It is slightly more convenient to work with a weighted version of the primes so we let 
$$\varpi(n)=\begin{cases}
\log n & n\text{ is prime}\\
0 & \text{Otherwise}.\\
\end{cases}$$
Finally, we adopt the standard convention that the value of $\epsilon$ may be different at each occurrence.  For example, we may write $x^{\epsilon}\log x\ll x^\epsilon$ and $x^{2\epsilon}\ll x^\epsilon$.  

We require the following forms of the Poisson Summation
Formula, which hold for all compactly supported smooth functions $f$,
all $v\in\R_{>0}$ and all $u\in \R$: 
\begin{equation}\label{pois1}
\sum_{m\in\Z}f(vm+u)=\frac{1}{v}\sum_{n\in \Z}\hat
f\left(\frac{n}{v}\right)e\left(\frac{un}{v}\right),
\end{equation}
\begin{equation}\label{pois2}
\sum_{n\in\Z}f\left(\frac{n}{v}\right)e\left(\frac{un}{v}\right)=v\sum_{m\in\Z}\hat f(vm-u).
\end{equation}

\section{Reduction of the Problem}

As we only require a lower bound we may smooth the function $\mathbf{1}_{\ca}$.

\begin{defn}\label{wdef}
Let $W$ be a smooth function satisfying the following conditions.

\begin{enumerate}
\item If $x\notin [\frac{1}{4},\frac{3}{4}]$ then $W(x)=0$.

\item If $x\in [\frac{1}{3},\frac{2}{3}]$ then $W(x)=1$.  

\item For all $x$, $0\leq W(x)\leq 1$.
\end{enumerate}
\end{defn}

It is a well known fact that many functions, $W$, satisfying the conditions of this definition exist.  The precise choice of $W$ does not matter but all implied constants may depend on it.  For any $B\in\N$ we may integrate by parts $B$ times to obtain the standard estimate 
\begin{equation}\label{fourierbound}
|\hat w(x)|\ll_B \min(1,|x|^{-B}).\\
\end{equation}

\begin{defn}\label{defPhi} 
Let 
$$\Phi(n)=\twosum{k}{n\equiv ka\pmod q}W(\frac{k}{z}).$$
\end{defn}

\begin{lem}
If $\frac{x}{4}<n<x$ then 
$$0\leq \Phi(n)\leq \mathbf{1}_{\ca}(n)\leq 1.$$
Therefore, to prove Theorem \ref{mainthm} it is sufficient to prove a lower bound for 
$$\twosum{\frac{x}{4}<n<x}{n\in \e_2}\Phi(n).$$
\end{lem}
\begin{proof}
This follows immediately from the definitions of $\ca$ and $\Phi$.
\end{proof}

\section{Type I Sums}

The Type I estimate we prove, Theorem \ref{typei},  has been known in essence since the work of Vaughan, \cite{vaughan}.  However, it is useful to
prove it again to get a result which is valid in our precise
situation.  In addition, Vaughan's proof uses estimates for
exponential sums whereas we use results from the geometry of numbers.
The exponential sum approach is possibly simpler for standard Type I
sums but we also need to estimate a variant of such sums, Theorem
\ref{typeiv},  which is easier with the geometry of numbers. 

Throughout this section $M,N\geq 1$ satisfy $\frac{x}{4}\leq MN\leq 4x$ and $M\leq z^{2-\delta}$ for some $\delta>0$.  This means that 
$$N\gg \frac{x}{M}\gg \frac{q}{z^{1-\delta}}.$$
All our implied constants may depend on $\delta$.

For an integer $m$ let 
$$\Psi(m)=\Psi(m;N)=\sum_{n\sim N}\Phi(mn).$$ 
We will consider $\Psi(m)$ as a counting function of points of a certain lattice, $\lambda(m)$.  

\begin{lem}
Let
$$\lambda(m)=\{(j,k)\in\Z^2:jq+ka\equiv 0\pmod m\}.$$
The set $\lambda(m)$ is a lattice in $\Z^2$ with determinant $m$.
\end{lem}
\begin{proof}
It is clear that $\lambda(m)$ is a lattice.  Since $(a,q)=1$ we know that $jq+ka$ takes on all integer values as $j,k$ vary over $\Z^2$.  Thus $jq+ka$ represents all congruence classes mod $m$ so the determinant of $\lambda(m)$ is $m$.
\end{proof}

Define $b_1(m)$ to be the shortest nonzero vector in $\lambda(m)$ and let $R_1(m)$ be the Euclidean length of $b_1(m)$.  We know, by Minkowski's Theorem, that $R_1(m)\ll \sqrt{m}$.

\begin{lem}\label{Psilem}
With the previous assumptions on $M,N,x,z$ and $q$ we have 
$$\Psi(m)=\frac{N\hat W(0)z}{q}+O(\frac{z}{R_1(m)}),$$
for any $m\sim M$.
\end{lem}
\begin{proof}
From the definitions of $\Psi$ and $\Phi$ we get 
\begin{eqnarray*}
\Psi(m)&=&\sum_{n\sim N}\Phi(mn)\\
&=&\sum_{n\sim N}\twosum{k}{mn\equiv ka\pmod q}W(\frac{k}{z})\\
&=&\sum_{n\sim N}\twosum{j,k}{mn=jq+ka}W(\frac{k}{z})\\
&=&\twosum{(j,k)\in\lambda(m)}{(jq+ka)/m\sim N}W(\frac{k}{z}).\\
\end{eqnarray*}
Since $W$ is supported on $(0,1)$ the sum only contains points with $k\in (0,z)$.  Let 
$$f(t)=\#\{(j,k)\in \lambda(m):\frac{jq+ka}{m}\sim N,k\in (0,t]\}.$$
Summing by parts we get 
$$\Psi(m)=-\frac{1}{z}\int_0^z f(t)W'(\frac{t}{z})\,dt.$$
Let 
$$A(t)=\{(x,y)\in \R^2:\frac{xq+ya}{m}\sim N,y\in (0,t]\}.$$
By a standard result for counting lattice points we have 
$$f(t)=\frac{\text{area}(A(t))}{m}+O(\frac{\text{perimeter}(A(t))}{R_1(m)}+1).$$
The vertices of $A(t)$ are 
$$(Nm/q,0),(2Nm/q,0),((Nm-ta)/q,t),((2Nm-ta)/q,t).$$
Therefore 
$$\text{area}(A(T))=\frac{Nmt}{q}$$
and
$$\text{perimeter}(A(T))\ll \frac{NM}{q}+t+\frac{ta}{q}\ll z.$$
It follows that
\begin{eqnarray*}
\Psi(m)&=&-\frac{1}{z}\int_0^z \left(\frac{Nt}{q}+O(\frac{z}{R_1(m)}+1)\right)W'(\frac{t}{z})\,dt\\
&=&-\frac{N}{qz}\int_0^ztW'(\frac{t}{z})\,dt+O(\frac{z}{R_1(m)}+1)\\
&=&\frac{N\hat W(0)z}{q}+O(\frac{z}{R_1(m)}+1).\\
\end{eqnarray*}
Since $R_1(m)\ll \sqrt{m}\ll\sqrt{M}\ll z$ the result follows.
\end{proof}

We need a bound for the number of $m$ for which $R_1(m)$ is unusually small.

\begin{lem}\label{lcount}
For any $\epsilon>0$, any $M\leq z^{2-\delta}$ and any integer $l$ we have
$$\#\{m\leq M:R_1(m)^2=l\}\ll_\epsilon  z^\epsilon.$$
\end{lem}
\begin{proof}
We know that $R_1(m)^2\ll m\ll M$.  Thus the only case to consider is $0<l\ll M$.

If $R_1(m)^2=l$ then there exist integers $j,k$ with $j^2+k^2=l$ and $m| jq+ka$.  It follows that the quantity of interest is bounded by 
$$\twosum{(j,k)\in\Z^2}{j^2+k^2=l}\#\{m:m|jq+ka\}\leq\twosum{(j,k)\in\Z^2}{j^2+k^2=l}\tau(jq+ka).$$
For the remainder of the proof let $h=jq+ka$, where $j^2+k^2=l$. We now use an argument by contradiction to show that $h\ne0$. If $h=0$ then
$k\ne0$ since $(j,k)\ne(0,0)$.  Moreover $q|k$, whence $|k|\ge q$.  However 
$$k\leq  \sqrt{l}\ll \sqrt{M}=o(z)=o(q),$$
giving a contradiction if $q$ is large enough.  We therefore
conclude that $h\ne0$.  In addition we have  
$$h\ll q\sqrt{l}\ll qz\ll x,$$
so that $\tau(h)\ll x^\epsilon$.  Letting $r(l)$ denote the number of ways in which $l$ may be written as the sum of two squares, the cardinality of the set in the lemma is then
$$\ll r(l)x^\epsilon\ll z^\epsilon,$$
in view of the convention on different values of $\epsilon$. 
\end{proof}

We may now prove an estimate for Type I sums.

\begin{thm}\label{typei}
If  $\alpha_m$ are  complex numbers with $|\alpha_m|\leq 1$  then, with the previous assumptions on $M,N,x,z$ and $q$ we have, for any $A>0$
$$\sum_{m\sim M,n\sim N}\alpha_m \Phi(mn)=\frac{\hat W(0)Nz}{q}\sum_{m\sim M}\alpha_m+O_A(z^2(\log z)^{-A}).$$
\end{thm}
\begin{proof}
Let
$$S=\sum_{m\sim M,n\sim N}\alpha_m \Phi(mn)=\sum_{m\sim M}\alpha_m \Psi(m).$$
Applying Lemma \ref{Psilem} we get 
$$S=\frac{N\hat W(0)z}{q}\sum_{m\sim M}\alpha_m+O(z\sum_{m\sim M}\frac{1}{R_1(m)}).$$
Using Lemma \ref{lcount} we deduce that 
\begin{eqnarray*}
\sum_{m\sim M}\frac{1}{R_1(m)}&=&\sum_{l\ll M}\frac{1}{\sqrt{l}}\#\{m\sim M,R_1(m)^2=l\}\\
&\ll_\epsilon&z^\epsilon\sum_{l\ll M}l^{-\frac12}\\
&\ll_\epsilon&z^\epsilon M^{\frac12}.\\
\end{eqnarray*}
We conclude that 
$$S=\frac{N\hat W(0)z}{q}\sum_{m\sim M}\alpha_m+O(z^{1+\epsilon}M^{\frac12}).$$
Since $M\ll z^{2-\delta}$ the error term is 
$$O(z^{2+\epsilon-\delta/2}).$$
The result follows on taking $\epsilon<\frac{\delta}{2}$.
\end{proof}

Observe that if 
$$\sum_{m\sim M}\alpha_m \asymp M$$
then the leading term in this estimate has size $\frac{xz}{q}\asymp z^2$.  This is  larger than the error term.

It is also necessary to bound a Type I sum where $\sum_{n\sim N}$ is replaced by a smooth weight.

\begin{thm}\label{typei2}
Suppose the above conditions on $M,N,x,z$ and $q$ hold. If  $\alpha_m$ are  complex numbers with $|\alpha_m|\leq 1$ then, for any $A>0$, we have 
$$\twosum{n}{m\sim M}\alpha_m W\left(\frac{n}{3N}\right)\Phi(mn)=\frac{3\hat W(0)^2Nz}{q}\sum_{m\sim M}\alpha_m+O_A(z^2(\log z)^{-A}).$$
\end{thm}
\begin{proof}
After using partial summation to remove the smooth weight $W(\frac{n}{3N})$, the result follows by an almost identical proof to that of Theorem \ref{typei}. 
\end{proof}

Define $\Psi_1(m)$ by 
$$\Psi(m)=\frac{\hat W(0)Nz}{q}+\Psi_1(m).$$
We will require the following two lemmas.

\begin{lem}
For any $\epsilon>0$ and any $M,N,x,q$ and $z$ satisfying the previous assumptions we have 
$$\sum_{m\asymp M}\Psi_1(m)^2\ll_\epsilon z^{2+\epsilon}.$$
\end{lem}
\begin{proof}
From Lemma \ref{Psilem} we have 
$$\Psi_1(m)^2\ll_\epsilon \frac{z^2}{R_1(m)^2}.$$
By Lemma \ref{lcount} we get
\begin{eqnarray*}
\sum_{m\asymp M}\frac{1}{R_1(m)^2}&\ll&\sum_{l\ll M}\frac{1}{l}\#\{m\asymp M,R_1(m)^2=l\}\\
&\ll_\epsilon&z^\epsilon\sum_{l\ll M}l^{-1}\\
&\ll_\epsilon&z^\epsilon.\\ 
\end{eqnarray*}
The result follows.
\end{proof}

\begin{lem}
Under the same assumptions as the last lemma we have 
$$\sum_{m\asymp M}\Psi(m)^2\ll \frac{Nz^3}{q}.$$
\end{lem}
\begin{proof}
We have 
\begin{eqnarray*}
\sum_{m\asymp M}\Psi(m)^2&=&\sum_{m\asymp M}(\frac{N\hat W(0)z}{q}+\Psi_1(m))^2\\
&\ll&\sum_{m\asymp M}\frac{N^2z^2}{q^2}+\sum_{m\asymp M}\psi_1(m)^2\\
&\ll_\epsilon&\frac{MN^2z^2}{q^2}+z^{2+\epsilon}\\
&\ll_\epsilon& \frac{Nz^3}{q}+z^{2+\epsilon}.\\
\end{eqnarray*}
Since $N\gg \frac{q}{z^{1-\delta}}$ the first term is larger if we take a small enough $\epsilon$.
\end{proof}

We may now estimate a variant of a Type I sum which will be useful later.

\begin{thm}\label{typeiv}
Suppose that the above assumptions on $M,N,x,z$ and $q$ hold. In addition, assume that $N\leq z^{2-\delta}$.   Then, for any complex numbers $\beta_n$ bounded by $1$ and any $A>0$,
$$\twosum{m}{n_1,n_2\sim N}\beta_{n_1}W\left(\frac{m}{3M}\right)
\Phi(mn_1)\Phi(mn_2)=\frac{3NM\hat W(0)^3z^2}{q^2}\sum_{n\sim N}\beta_n+O_A(\frac{z^4(\log z)^{-A}}{M}).$$
\end{thm}
\begin{proof}
Let 
$$S=\twosum{m}{n_1,n_2\sim N}\beta_{n_1}W\left(\frac{m}{3M}\right)\Phi(mn_1)\Phi(mn_2)=\twosum{m}{n\sim N}\beta_nW\left(\frac{m}{3M}\right)\Phi(mn)\Psi(m).$$
Writing 
$$\Psi(m)=\frac{\hat W(0)Nz}{q}+\Psi_1(m)$$
we get a contribution from $\frac{\hat W(0)Nz}{q}$ of 
$$\frac{\hat W(0)Nz}{q}\twosum{m}{n\sim N}\beta_nW\left(\frac{m}{3M}\right)\Phi(mn).$$
This sum is in a form which can be estimated by Theorem \ref{typei2},
with $m,n$ interchanged. All the conditions needed for that theorem are satisfied since $N\leq z^{2-\delta}$.  The main term is thus  
$$\frac{3\hat W(0)^3NMz^2}{q}\sum_{n\sim N}\beta_n+O_A(\frac{z^3(\log z)^{-A}N}{q}).$$
On writing $N\ll\frac{zq}{M}$ the error here is 
$$O_A(\frac{z^4(\log z)^{-A}}{M}).$$ 

The contribution from $\Psi_1(m)$ is 
$$\twosum{m}{n\sim N}\beta_nW\left(\frac{m}{3M}\right)\Phi(mn)\Psi_1(m).$$
Trivially estimating the $\beta_n$ by $1$ this is majorised by 
$$\sum_mW\left(\frac{m}{3M}\right)\Psi(m)|\Psi_1(m)|.$$
Since $W(x)\leq 1$ for all $x$ we may remove the factor $W(\frac{m}{3M})$ and apply Cauchy's inequality to get a bound of 
$$(\sum_{m\asymp M}\Psi(m)^2)^{1/2}(\sum_{m\asymp M}\Psi_1(m)^2)^{1/2}.$$
Applying the previous two lemmas this is 
$$\ll_\epsilon N^{1/2}z^{5/2+\epsilon}q^{-1/2}\ll \frac{z^{3+\epsilon}}{\sqrt M}.$$
Since $M\leq z^{2-\delta}$ the error here is 
$$\frac{z^{3+\epsilon}\sqrt{M}}{M}\leq \frac{z^{4+\epsilon-\delta/2}}{M}.$$
The result follows on taking $\epsilon<\frac{\delta}{2}$.
\end{proof}

Observe that if 
$$\sum_{n\sim N}\beta_n\asymp N$$
then the main term in this last theorem has size 
$$\frac{N^2Mz^2}{q^2}\asymp \frac{z^4}{M}.$$

\section{Type II Sums}

We will prove the following Type II result.

\begin{thm}\label{typeii}
Let $\alpha_m$ be complex numbers bounded by $1$.  Suppose that  $\frac{x}{4}\leq MN\leq 4x$ and 
$$\max(z,\frac{q}{z^{1-\delta}})\leq N\leq z^{\frac{16}{15}-\delta}$$ 
for some $\delta>0$.  Then,  for every $A>0$, we have 
$$\sum_{m\sim M,n\sim N}\alpha_m(\varpi(n)-1)\Phi(mn)\ll z^2(\log z)^{-A},$$
where the implied constant depends on both $A$ and $\delta$.  
\end{thm}

Observe that the restrictions on $M,N$ in this theorem imply that 
$$M\ll z^{2-\delta}.$$
The hypothesis that $N\geq z$ is only used once in our
argument, in the proof of Lemma \ref{zerosum}.  When $\tau>\frac13$ this assumption is weaker than 
$$N\geq  \frac{q}{z^{1-\delta}}.$$

Let 
$$S=\sum_{m\sim M,n\sim N}\alpha_m\beta_n\Phi(mn),$$
where 
$$\beta_n=\varpi(n)-1.$$
We wish to show that  $S=O(z^2(\log z)^{-A})$. Our arguments can be
modified to handle arbitrary $\beta_n$, although the range of $N$ is
then much smaller.  However, this introduces some additional
technicalities.  Since our Type II estimate does not cover a
sufficiently large range of $N$ to detect primes we have chosen to  give the details only for the specific choice $\beta_n=\varpi(n)-1$.   

Vaughan, \cite{vaughan},  used exponential sum methods  to establish Type II estimates which are only valid when $x^\tau<N<x^{1-2\tau}$.  This range is empty when $\tau\geq\frac{1}{3}$.  Heath-Brown and Jia, \cite{rhbjia}, introduced a new method which reduces the problem to the estimation of certain Kloosterman sums.  Matom{\"a}ki, \cite{mat},  used the same reduction but then used stronger bounds on the resulting averages of Kloosterman sums and was thus able to get enough Type II information to detect primes for any $\tau<\frac{1}{3}$.  The range of $N$ in the Type II bounds found by Heath-Brown, Jia  and Matom{\"a}ki remains nonempty as $\tau\rightarrow\frac{1}{3}$.  However, it is not valid for $\tau\geq\frac{1}{3}$ as the reduction to Kloosterman sums gives an error which is too large in this case.  Our method is essentially an extension of that of Heath-Brown and Jia which avoids this problem.

\begin{lem}
We have $S=O(\sqrt{MS_1})$ where 
$$S_1=\sum_{n_1,n_2\sim N}\beta_{n_1}\beta_{n_2}\sum_m W\left(\frac{m}{3M}\right)\Phi(mn_1)\Phi(mn_2).$$
It follows that a bound of 
$$S_1=O(\frac{z^4(\log z)^{-A}}{M}).$$
will be sufficient.
\end{lem}
\begin{proof}
Applying Cauchy's inequality gives 
$$S^2\leq\sum_{m\sim M}|\alpha_m|^2\sum_{m\sim M}(\sum_{n\sim N}\beta_n \Phi(mn))^2.$$
By definition of the function $W$ we know that $W\left(\frac{m}{3M}\right)=1$ when $m\sim M$.  Therefore 
\begin{eqnarray*}
S^2&\ll&M\sum_mW\left(\frac{m}{3M}\right)(\sum_{n\sim N}\beta_n\Phi(mn))^2\\  
&=&M\sum_{n_1,n_2\sim N}\beta_{n_1}\beta_{n_2}\sum_m W\left(\frac{m}{3M}\right)\Phi(mn_1)\Phi(mn_2)\\
&=&MS_1.\\
\end{eqnarray*}
\end{proof}

On putting $\beta_n=\varpi(n)-1$   into $S_1$ we will get three sums all of which must be evaluated asymptotically.  However, on combining the sums, all the main terms will cancel and we will get the required result.  Specifically let 
$$S_1=S_{1,1}-2S_{1,2}+S_{1,3}$$
where 
$$S_{1,1}=\sum_{n_1,n_2\sim N}\varpi(n_1)\varpi(n_2)\sum_m W\left(\frac{m}{3M}\right)\Phi(mn_1)\Phi(mn_2),$$
$$S_{1,2}=\sum_{n_1,n_2\sim N}\varpi(n_1)\sum_m W\left(\frac{m}{3M}\right)\Phi(mn_1)\Phi(mn_2)$$
and
$$S_{1,3}=\sum_{n_1,n_2\sim N}\sum_m W\left(\frac{m}{3M}\right)\Phi(mn_1)\Phi(mn_2).$$

We begin by dealing with the sums $S_{1,2}$ and $S_{1,3}$.

\begin{lem}
With our assumptions on $M,N,x,q$ and $z$ we have, for $i=2,3$ that
$$S_{1,i}=\frac{3N^2M\hat W(0)^3z^2}{q^2}+O_A\left(\frac{z^4(\log z)^{-A}}{M}\right).$$
\end{lem}
\begin{proof}
We have 
$$N\leq z^{\frac{16}{15}-\delta}\leq z^{2-\delta}.$$
We may therefore use Theorem \ref{typeiv} with $\beta_n=\varpi(n)$ or $\beta_n=1$.  These coefficients are only bounded by $\log n$ but this can be absorbed into the error term.  In either case we have 
$$\sum_{n\sim N}\beta_n=N+O(N(\log N)^{-A})$$
so the result follows.
\end{proof}

Next we deal with the contribution to $S_{1,1}$ from pairs with $n_1=n_2$.  This is 
$$\sum_{n\sim N}\varpi(n)^2\sum_m W\left(\frac{m}{3M}\right)\Phi(mn)^2.$$
All the terms are positive and $\Phi$ takes values in $[0,1]$ so this is at most
$$\sum_{n\sim N}\varpi(n)^2\sum_m W\left(\frac{m}{3M}\right)\Phi(mn).$$
Using Theorem \ref{typei2} we may bound this Type I sum by $O(z^2\log N)$.  Since $M\ll z^{2-\delta}$ this is $O(\frac{z^4(\log z)^{-A}}{M})$.

The remaining terms in $S_{1,1}$ have $n_1\ne n_2$.  Since the coefficients $\varpi(n)$ are supported on primes all such pairs actually satisfy $(n_1,n_2)=1$. We therefore consider 
$$S_2=\twosum{n_1,n_2\sim N}{(n_1,n_2)=1}\varpi(n_1)\varpi(n_2)\sum_m W\left(\frac{m}{3M}\right)\Phi(mn_1)\Phi(mn_2).$$

\subsection{Harmonic Analysis of the Sum $S_{2}$}

Let 
$$T=\sum_m W\left(\frac{m}{3M}\right)\Phi(mn_1)\Phi(mn_2).$$
Since $(a,q)=1$ there exists an $\overline a$ satisfying 
$$a\overline a\equiv 1\pmod q.$$

\begin{lem}
We have 
$$T=\frac{3Mz^2}{q^2}\sum_{k_1,k_2}\hat W(\frac{k_1z}{q})\hat W(\frac{k_2z}{q})\sum_m\hat W\left(3M\left(m-\frac{\overline a(k_1n_1+k_2n_2)}{q}\right)\right).$$
\end{lem}
\begin{proof}
The definition of $\Phi$ gives
\begin{eqnarray*}
\Phi(n)&=&\twosum{k}{k\equiv n\overline a\pmod q}W(\frac{k}{z})\\
&=&\sum_m W\left(\frac{qm+n\overline a}{z}\right).\\
\end{eqnarray*}
Applying the Poisson Summation Formula in the form (\ref{pois1}) we therefore get 
$$\Phi(n)=\frac{z}{q}\sum_k\hat W\left(\frac{kz}{q}\right)e\left(\frac{n\overline ak}{q}\right),$$
so that 
$$T=\frac{z^2}{q^2}\sum_{k_1,k_2} \hat W\left(\frac{k_1z}{q}\right)\hat W\left(\frac{k_2z}{q}\right)\sum_mW\left(\frac{m}{3M}\right)e\left(\frac{m\overline a(k_1n_1+k_2n_2)}{q}\right).$$
We can now use the Poisson Summation Formula (\ref{pois2}) to obtain 
$$\sum_mW\left(\frac{m}{3M}\right)e\left(\frac{m\overline a(k_1n_1+k_2n_2)}{q}\right)=3M\sum_m\hat W\left(3Mm-\frac{3M\overline a(k_1n_1+k_2n_2)}{q}\right).$$ 
The result follows on substituting this into the above expression for $T$.
\end{proof}

Let $S_3$ be the subsum of $S_2$ coming from terms with $k_1n_1+k_2n_2=0$.  Since $(n_1,n_2)=1$ any solution of this may be written uniquely as $k_1=n_2h$ and $k_2=-n_2h$ for some $h\in \Z$.  Therefore 
$$S_3=\frac{3Mz^2}{q^2}\twosum{n_1,n_2\sim N}{(n_1,n_2)=1}\varpi(n_1)\varpi(n_2)\sum_{h,m}\hat W(\frac{n_2hz}{q})\hat W(\frac{-n_1hz}{q})\hat W(3Mm).$$

\begin{lem}
For any $A>0$ we have, under the previous assumptions on $M,N,x,q$ and $z$, that  
$$S_3=\frac{3MN^2z^2\hat W(0)^3}{q^2}+O_A(\frac{z^4(\log z)^{-A}}{M}).$$
\end{lem}
\begin{proof}
Our assumptions imply that for $n_i\sim N$ we have 
$$\frac{n_iz}{q}\gg \frac{Nz}{q}\gg z^\delta$$
and that 
$$M\gg z^\delta.$$
It follows, using the bound (\ref{fourierbound}),  that the contribution to $S_3$ from terms with $h\ne 0$ or $m\ne 0$ is negligible.  Specifically, for any $B\in\N$ we have 
$$S_3=\frac{3Mz^2\hat W(0)^3}{q^2}\twosum{n_1,n_2\sim N}{(n_1,n_2)=1}\varpi(n_1)\varpi(n_2)+O_B(z^{-B}).$$
Observe that 
$$\frac{Mz^2}{q^2}\sum_{n\sim N}\varpi(n)^2\ll \frac{MNz^2\log N}{q^2}\ll \frac{z^3\log N}{q}\ll_A \frac{z^4(\log z)^{-A}}{M},$$
where the last inequality uses that $M\ll z^{2-\delta}\leq qz^{1-\delta}$.  We deduce that 
$$S_3=\frac{3Mz^2\hat W(0)^3}{q^2}\sum_{n_1,n_2\sim N}\varpi(n_1)\varpi(n_2)+O_A(\frac{z^4(\log z)^{-A}}{M}).$$
The result follows on applying the Prime Number Theorem to the sum 
$$\sum_{n\sim N}\varpi(n).$$
\end{proof}

Let $S_4$ be the sum of the remaining terms from $S_2$, those with $k_1n_1+k_2n_2\ne 0$.  Thus 
$$S_4=\twosum{n_1,n_2\sim N}{(n_1,n_2)=1}\varpi(n_1)\varpi(n_2)T_1,$$
where
$$T_1=\frac{3Mz^2}{q^2}\twosum{k_1,k_2}{k_1n_1+k_2n_2\ne 0}\hat W(\frac{k_1z}{q})\hat W(\frac{k_2z}{q})\sum_m\hat W\left(3M\left(m-\frac{\overline a(k_1n_1+k_2n_2)}{q}\right)\right).$$
For any integers $m,k_1,k_2$ there exists a unique integer $k$ such that 
$$m-\frac{\overline a(k_1n_1+k_2n_2)}{q}=\frac{k}{q}.$$
There is then a unique integer $j$ such that 
$$k_1n_1+k_2n_2=jq-ka.$$
Writing $c=jq-ka$ it follows that 
$$T_1=\frac{3Mz^2}{q^2}\twosum{j,k,k_1,k_2}{k_1n_1+k_2n_2=c\ne 0}\hat W(\frac{k_1z}{q})\hat W(\frac{k_2z}{q})\hat W(\frac{3Mk}{q}).$$
If we let 
$$F(n_1,n_2;c)=\twosum{k_1,k_2}{k_1n_1+k_2n_2=c}\hat W(\frac{k_1z}{q})\hat W(\frac{k_2z}{q})$$
then 
$$S_4=\frac{3Mz^2}{q^2}\twosum{j,k}{c\ne 0}\hat W(\frac{3Mk}{q})\twosum{n_1,n_2\sim N}{(n_1,n_2)=1}\varpi(n_1)\varpi(n_2)F(n_1,n_2;c).$$

\subsection{Transforming the Function $F$}

To deal with the sum $S_4$ we begin by applying Poisson Summation to the function $F$.

\begin{lem}\label{Fpois}
Let $\overline n_1$ be an inverse of $n_1$ modulo $n_2$, which exists since $(n_1,n_2)=1$.  We have
$$F(n_1,n_2)=\frac{1}{n_2}\sum_l \hat g\left(\frac{l}{n_2};n_1,n_2,c\right)e\left(\frac{c\overline n_1 l}{n_2}\right),$$
where 
$$g(t;n_1,n_2,c)=\hat W\left(\frac{tz}{q}\right)\hat W\left(\frac{(c-tn_1)z}{n_2q}\right)$$
and $\hat g$ is the Fourier transform of $g$ with respect to the single variable $t$.
\end{lem}

\begin{proof}
We are interested in pairs $k_1,k_2$ satisfying the equation
$$k_1n_1+k_2n_2=c.$$
For a given $k_1$ this  has at most $1$ solution which exists if and only if 
$$k_1n_1\equiv c\pmod {n_2}.$$
Since $(n_1,n_2)=1$ this condition is equivalent to 
$$k_1\equiv c\overline{n_1}\pmod {n_2}.$$
If this congruence holds then the corresponding $k_2$ is given by 
$$k_2=\frac{c-k_1n_1}{n_2}.$$
We therefore have 
$$F(n_1,n_2)=\twosum{k}{k\equiv c\overline{n_1}\pmod{n_2}}\hat W\left(\frac{kz}{q}\right)\hat W\left(\frac{(c-kn_1)z}{n_2q}\right).$$
Now, if we let 
$$g(t;n_1,n_2,c)=\hat W\left(\frac{tz}{q}\right)\hat W\left(\frac{(c-tn_1)z}{n_2q}\right),$$
then by the Poisson Summation Formula, (\ref{pois1}), we get 
$$F(n_1,n_2)=\frac{1}{n_2}\sum_l \hat g\left(\frac{l}{n_2}\right)e\left(\frac{c\overline n_1 l}{n_2}\right).$$
\end{proof}

Applying this lemma to the sum $S_4$ we deduce that 
$$S_4=\frac{3Mz^2}{q^2}\twosum{j,k,l}{c\ne 0}\hat W(\frac{3Mk}{q})\twosum{n_1,n_2\sim N}{(n_1,n_2)=1}\varpi(n_1)\varpi(n_2)\frac{1}{n_2}\hat g\left(\frac{l}{n_2};n_1,n_2,c\right)e\left(\frac{c\overline n_1 l}{n_2}\right).$$
The sums considered by Heath-Brown and Jia, as well as by Matom{\"a}ki, are essentially just the $k=0$ terms of $S_4$.

\subsection{Terms with $l=0$}

We will need the following result concerning the function $\hat g$.

\begin{lem}\label{ltrunc}
For all $t$ and all $n_1,n_2\sim N$ we have $\hat g(t)\ll
\frac{q}{z}$.  Furthermore, if $|t|\geq \frac{4z}{q}$ then $\hat g(t)=0$. 
\end{lem}
\begin{proof}
Recall that 
$$g(t)=\hat W\left(\frac{tz}{q}\right)\hat W\left(\frac{(c-tn_1)z}{n_2q}\right)=g_1(t)g_2(t),$$
say. It follows that 
$$\hat g(t)=(\hat g_1\star \hat g_2)(t)=\int_{-\infty}^\infty \hat g_1(x)\hat g_2(t-x)\,dx.$$
We have 
$$g_1(t)=\hat W(\frac{tz}{q})$$
so
$$\hat g_1(t)=\frac{q}{z}W(-tq/z).$$
We also have 
$$g_2(t)=\hat W\left(\frac{(c-tn_1)z}{n_2q}\right)$$
so
$$\hat g_2(t)=\frac{n_2q}{n_1z}W(\frac{n_2qt}{n_1z})e(\frac{-ctq}{n_2z}).$$
Therefore, for all $t$ we deduce that
$$|\hat g_i(t)|\ll  \frac{q}{z}.$$
Furthermore, if $|t|\geq \frac{2z}{q}$, then
$$\hat g_i(t)=0.$$

It follows that for all $t$ we have 
$$\hat g(t)=\int_{-\infty}^\infty \hat g_1(x)\hat g_2(t-x)\,dx\ll
\int_{|x|\leq \frac{2z}{q}}(q/z)^2\,dx\ll \frac{q}{z}.$$ 
In addition, if $|t|\geq \frac{4z}{q}$ 
then for any $x$ either 
$$|x|\geq \frac{2z}{q}$$
or 
$$|t-x|\geq \frac{2z}{q}.$$
It follows that $\hat g(t)=0$.  
\end{proof}

Let $S_5$ be the subsum of $S_4$ containing the terms with $l=0$, that is 
$$S_5=\frac{3Mz^2}{q^2}\twosum{j,k}{c\ne 0}\hat
W(\frac{3Mk}{q})\twosum{n_1,n_2\sim
  N}{(n_1,n_2)=1}\varpi(n_1)\varpi(n_2)\frac{1}{n_2}\hat
g(0;n_1,n_2,c).$$ 
It is convenient to reinstate the terms with $c=0$.  These correspond
to pairs $(j,k)$ with $k=hq,j=ha$ 
so their contribution is  
$$\frac{3Mz^2}{q^2}\sum_h\hat W(3Mh)\twosum{n_1,n_2\sim N}{(n_1,n_2)=1}\varpi(n_1)\varpi(n_2)\frac{1}{n_2}\hat g(0;n_1,n_2,0).$$
From the estimate (\ref{fourierbound}) we may deduce that for any $B\in\N$ the contribution to this from terms with $h\ne 0$ is $O_B(z^{-B})$.  Using the estimate  for $\hat g$ given in Lemma \ref{ltrunc} we may bound the $h=0$ terms by 
$$\frac{MNz}{q}\ll z^2\ll_A \frac{z^4(\log z)^{-A}}{M},$$
since $M\ll z^{2-\delta}$.  It is therefore enough to bound 
$$S_6=\frac{3Mz^2}{q^2}\sum_{j,k}\hat W(\frac{3Mk}{q})\twosum{n_1,n_2\sim N}{(n_1,n_2)=1}\varpi(n_1)\varpi(n_2)\frac{1}{n_2}\hat g(0;n_1,n_2,c).$$
We may move the sum over $j$ inside the other summations to transform this to 
$$S_6=\frac{3Mz^2}{q^2}\sum_k\hat W(\frac{3Mk}{q})\twosum{n_1,n_2\sim N}{(n_1,n_2)=1}\varpi(n_1)\varpi(n_2)\frac{1}{n_2}\sum_j\hat g(0;n_1,n_2,c).$$
Inserting the definition of $\hat g$ and reordering we see that 
$$S_6=\frac{3Mz^2}{q^2}\sum_k\hat W(\frac{3Mk}{q})\twosum{n_1,n_2\sim N}{(n_1,n_2)=1}\varpi(n_1)\varpi(n_2)\frac{1}{n_2}\int_{-\infty}^\infty\hat W\left(\frac{tz}{q}\right)\sum_j\hat W\left(\frac{(c-tn_1)z}{n_2q}\right)\,dt.$$

\begin{lem}\label{zerosum}
For all $t\in\R$, $N\geq z$ and
$n_1,n_2\sim N$ we have 
$$\sum_j\hat W\left(\frac{(c-tn_1)z}{n_2q}\right)=0.$$
\end{lem}
\begin{proof}
The sum is 
$$\sum_j\hat W\left(\frac{(jq-ka-tn_1)z}{n_2q}\right).$$
We may  apply the Poisson Summation Formula, (\ref{pois1}), to obtain
$$\frac{n_2}{z}\sum_j W(\frac{n_2j}{z})e(\gamma j),$$
for a $\gamma$ which depends on all the outer variables.

Since $N\geq z$ we have 
$$\frac{n_2}{z}\geq \frac{N}{z}\geq 1.$$ 
However, $W$ is supported on $[\frac{1}{4},\frac{3}{4}]$ and thus for all $n\in \N$ we have 
$$W(\frac{n_2n}{z})=0.$$
\end{proof}

It follows from this that $S_6=0$ and therefore that 
$$S_5\ll_A\frac{z^4(\log z)^{-A}}{M}.$$

\subsection{The Remaining Terms}

Let $S_7$ be the subsum of $S_4$ containing all the remaining terms,
that is to say, all those with $l\ne 0$.  Thus  
$$S_7=\frac{3Mz^2}{q^2}\twosum{j,k,l}{c\ne 0,l\ne 0}\hat W(\frac{3Mk}{q})\twosum{n_1,n_2\sim N}{(n_1,n_2)=1}\varpi(n_1)\varpi(n_2)\frac{1}{n_2}\hat g\left(\frac{l}{n_2};n_1,n_2,c\right)e\left(\frac{c\overline n_1 l}{n_2}\right).$$
We now truncate the sums over $j,k,l$ to finite ranges.  

\begin{lem}
Suppose $\eta>0$.  
The contribution to $S_7$ from $(j,k,l)$ for
which any of  
$$|l|\geq \frac{8Nz}{q},$$ 
$$|k|\geq \frac{qz^{\eta}}{M}$$
or 
$$|j|\geq Nz^{-1+2\eta}$$
hold is $O_{B,\eta}(z^{-B})$ for any $B\in\N$.  
\end{lem}
\begin{proof}
From Lemma \ref{ltrunc} we know that if $|t|\geq \frac{4z}{q}$ then
$\hat g(t)=0$.  It follows that terms with 
$$|l|\geq \frac{8Nz}{q}$$
make no contribution to the sum.  

Let $R$ be the set of $(j,k)$ for which 
$$|k|\geq \frac{qz^{\eta}}{M}$$
or 
$$|j|\geq Nz^{-1+2\eta}.$$
To complete the proof it is sufficient to give a bound of $O_B(z^{-B})$ for 
$$\sum_{(j,k)\in R}\left|\hat W(\frac{3Mk}{q})\hat g\left(\frac{l}{n_2};n_1,n_2,c\right)\right|.$$
By definition of $\hat g$ this is at most 
$$\int_{-\infty}^\infty\sum_{(j,k)\in R}\left|\hat W(\frac{3Mk}{q})\hat W\left(\frac{tz}{q}\right)\hat W\left(\frac{(jq-ka-tn_1)z}{n_2q}\right)\right|\,dt.$$
We make repeated use of the estimate (\ref{fourierbound}).  This shows that any part of the above where $\hat W$ is evaluated at a point $x$ with $|x|\geq z^{\eta}$ may be bounded by $O_B(z^{-B})$.  From the factor 
$\hat W(\frac{tz}{q})$
we see that such a bound holds when 
$$|t|\geq \frac{q}{z^{1-\eta}}$$
and from the factor 
$\hat W(\frac{3Mk}{q})$
it holds when 
$$|k|\geq  \frac{qz^{\eta}}{M}.$$
Finally we assume that 
$$|t|< \frac{q}{z^{1-\eta}}$$
and
$$|k|<  \frac{qz^{\eta}}{M}.$$
In this case we have 
$$|j|\geq Nz^{-1+2\eta}.$$
For sufficiently large $q$ these assumptions imply that 
$$\frac{(jq-ka-tn_1)z}{n_2q}\gg z^{\eta}.$$
A bound of $O_B(z^{-B})$ therefore holds for all parts of the sum.
\end{proof}

Let $S_8$ be the sum $S_7$ with the following ranges of summation:
$$0<|l|<\frac{8Nz}{q},$$
$$|k|< \frac{qz^{\eta}}{M}$$
and
$$|j|< Nz^{-1+2\eta}.$$
The last lemma shows that, for a fixed $\eta>0$,  we only need to bound $S_8$.  We ignore any potential cancellation in the outer sums so we write 
$$S_8\ll \frac{Mz^2\log  N}{q^2N}
\twosum{|j|<Nz^{-1+2\eta},\,|k|<\frac{qz^{\eta}}{M},\,0<|l|<\frac{8Nz}{q}}{c\ne0}S_9$$ 
where 
$$S_9=\sum_{n_2\sim N}|\twosum{n_1\sim N}{(n_1,n_2)=1}\varpi(n_1)\hat
g\left(\frac{l}{n_2};n_1,n_2,c\right)e\left(\frac{c\overline n_1
    l}{n_2}\right)|.$$ 
Let $h(n_1,n_2)$ be the weight in this sum: 
$$h(n_1,n_2)=\hat g\left(\frac{l}{n_2}\right)=\int_{-\infty}^\infty \hat W\left(\frac{tz}{q}\right)\hat W\left(\frac{(c-tn_1)z}{n_2q}\right)e(-\frac{tl}{n_2})\,dt.$$

\begin{lem}
The function $h$ depends smoothly on $n_1$ and $n_2$.  For
$n_1,n_2\sim N$ and the same $\eta$ as above, we have  
$$h(n_1,n_2)\ll \frac{q}{z}$$
and
$$h_{n_1}(n_1,n_2)\ll_{\eta} \frac{q}{Nz^{1-\eta}}.$$
\end{lem}
\begin{proof}
Since $W$ is smooth, it follows that $g$ depends smoothly on $n_1,n_2$ and therefore so does $\hat g$ and hence so does $h$.  The bound for $h$ follows from that for $\hat g$ given in Lemma \ref{ltrunc}.  

Differentiating we get 
$$h_{n_1}(n_1,n_2)=\int_{-\infty}^\infty \hat W\left(\frac{tz}{q}\right)\frac{-tz}{n_2q}\hat W'\left(\frac{(c-tn_1)z}{n_2q}\right)e(-\frac{tl}{n_2})\,dt.$$
The contribution to the integral from $|t|\geq \frac{q}{z^{1-\eta/2}}$ can be shown to be sufficiently small.  The remainder of the integral is then bounded by 
$$\int_{|t|\leq \frac{q}{z^{1-\eta/2}}}\frac{tz}{Nq}\,dt\leq \int_{|t|\leq \frac{q}{z^{1-\eta/2}}}\frac{z^{\eta/2}}{N}\,dt\ll \frac{q}{Nz^{1-\eta}}.$$
\end{proof}

We may now use partial summation to remove the weight $h(n_1,n_2)$ from $S_9$.  We deduce that
$$S_9\ll_{\eta} \frac{q}{z^{1-\eta}}S_{10}$$
where 
$$S_{10}=\max_{N'\sim N}\sum_{n_2\sim N}|\twosum{N\leq n_1<N'}{(n_1,n_2)=1}\varpi(n_1)e\left(\frac{c\overline n_1 l}{n_2}\right)|.$$
We will estimate $S_{10}$ using our bound, \cite[Theorem 1.3]{mykloost}.  For any $\epsilon>0$ this gives 
$$S_{10}\ll_\epsilon\left(1+\frac{|cl|}{N^2}\right)^{\frac12}N^{2-\alpha-\epsilon},$$
with the specific value $\alpha=\frac18$.  Since 
$$0<|cl|\ll N^2z^{2\eta}$$
we deduce that 
$$S_{10}\ll_\epsilon z^{\eta}N^{2-\alpha+\epsilon}.$$
We will eventually choose $\eta$ in such a way that the factor $z^\eta$ in this bound has no effect on the quality of our final result.  It is the value of $\alpha$ which determines the size of the admissible range for $N$ and hence the limitation on $\tau$.

\begin{lem}
Under the previous assumptions on $M,N,x,q$ and $z$ we have 
$$S_7\ll_A \frac{z^4(\log z)^{-A}}{M},$$
for any fixed $A>0$.
\end{lem}
\begin{proof}
We deduce from our bound for $S_{10}$ that 
$$S_9\ll_\epsilon \frac{q}{z^{1-2\eta}}N^{2-\alpha+\epsilon}$$
and therefore that 
$$S_8\ll_\epsilon \frac{N^{3-\alpha}z^{1+5\eta+\epsilon}}{q}.$$
By assumption we have 
$$N\leq z^{\frac{16}{15}-\delta}=z^{\frac{2}{2-\alpha}-\delta}.$$
It follows that
\begin{eqnarray*}
MS_8&\ll_\epsilon& \frac{MN^{3-\alpha}z^{1+5\eta+\epsilon}}{q}\\
&\ll& N^{2-\alpha}z^{2+5\eta+\epsilon}\\
&\leq& z^{4-\delta(2-\alpha)+5\eta+\epsilon}.\\
\end{eqnarray*}
We can choose $\epsilon,\eta$ sufficiently small so that 
$$5\eta+\epsilon< \delta(2-\alpha),$$
whence
$$S_8\ll_{\delta}\frac{z^4(\log z)^{-A}}{M}.$$
The bound for $S_7$ follows.
\end{proof}

Recall that we are assuming  $N\gg \frac{q}{z^{1-\delta}}$.  Observe that 
$$\frac{q}{z}<z^{\frac{2}{2-\alpha}}$$
if and only if 
$$q<z^{\frac{4-\alpha}{2-\alpha}}.$$
We note that 
$$\frac{4-\alpha}{2-\alpha}\frac{1-\tau}{1+\tau}>1$$ 
if and only if $\tau<\frac{1}{3-\alpha}=\frac{8}{23}$.  We therefore impose the condition $\tau<\frac{8}{23}$ in order to ensure that our range
for $N$ is nonempty.

It should be noted that in this section we have made nontrivial use of
the fact that our coefficients are the indicator function of the
primes.  If we want to estimate a general Type II sum with
coefficients $\beta_n$ then different bounds must be used.
Specifically, if we use Duke, Friedlander and Iwaniec's result,
\cite[Theorem 2]{dfi}, then we can take $\alpha=\frac{1}{48}$.  This
is much worse than the value $\frac{1}{8}$ which we have for our
special coefficients; although even that is considerably weaker than
$\alpha=\frac{1}{2}$, which we conjecture should be best possible. 

\subsection{Completing the Proof of Theorem \ref{typeii}}

The result follows on combining all the above estimates.  We have 
\begin{eqnarray*}
S_1&=&S_{1,1}-2S_{1,2}+S_{1,3}\\
&=&S_{1,1}-\frac{3N^2M\hat W(0)^3z^2}{q^2}+O_A\left(\frac{z^4(\log z)^{-A}}{M}\right)\\
&=&S_2-\frac{3N^2M\hat W(0)^3z^2}{q^2}+O_A\left(\frac{z^4(\log z)^{-A}}{M}\right)\\
&=&S_3+S_4-\frac{3N^2M\hat W(0)^3z^2}{q^2}+O_A\left(\frac{z^4(\log z)^{-A}}{M}\right)\\
&=&S_4+O_A\left(\frac{z^4(\log z)^{-A}}{M}\right)\\
&=&S_5+S_7+O_A\left(\frac{z^4(\log z)^{-A}}{M}\right)\\
&=&O_A\left(\frac{z^4(\log z)^{-A}}{M}\right).\\
\end{eqnarray*}
It follows that 
$$S=O_A(z^2(\log z)^{-A}),$$
as required.

\section{Proof of the Theorems}\label{proofs}
\subsection{Proof of Theorem \ref{mainthm}}

Suppose $MN=\frac{x}{4}$ and $M\leq z^{2-\delta}$, for some $\delta>0$.  For any $A>0$ we have 
$$\sum_{m\sim M}\varpi(m)=M+O_A(M(\log M)^{-A}).$$
It follows by Theorem \ref{typei} that 
$$\sum_{m\sim M,n\sim N}\varpi(m)\Phi(mn)=\frac{\hat W(0)}4z^2+O_{\delta,A}(z^2(\log z)^{-A});$$
the fact that $\varpi(n)$ is only bounded by $\log n$ does not matter as this factor can be absorbed into the error term.

Suppose, in addition, that   
$$\max(z,\frac{q}{z^{1-\delta}})\leq N\leq z^{\frac{16}{15}-\delta}.$$  
It follows from Theorem \ref{typeii} that for any $A>0$ we have 
$$\sum_{m\sim M,n\sim N}\varpi(m)(\varpi(n)-1)\Phi(mn)\ll_{A,\delta} z^2(\log z)^{-A}.$$
Combining these two estimates we immediately deduce that
$$\sum_{m\sim M,n\sim N}\varpi(m)\varpi(n)\Phi(mn)=\frac{\hat W(0)}{4}z^2+O_{A,\delta}(z^2(\log z)^{-a}).$$
If $m$ and $n$ are prime then $\varpi(m)\varpi(n)\asymp (\log z)^2$.  It follows that for sufficiently large $q$ we have 
$$\twosum{m\sim M,n\sim N}{mn\in\e_2}\Phi(mn)\gg \frac{z^2}{(\log z)^2}.$$
For $\tau<\frac{8}{23}$ there are exponents $a(\tau)<b(\tau)$ such that the above bound holds for any range $(M,2M]\subseteq(z^{a(\tau)},z^{b(\tau)}]$.  There are therefore $\gg_{\tau} \log z$ dyadic ranges available. Theorem
\ref{mainthm} follows.  

\subsection{Proof of Theorem \ref{diophantine}}

Suppose $\alpha$ is irrational and $\tau<\frac{8}{23}$.  By replacing $\tau$ by $\tau+\epsilon$ for a sufficiently small $\epsilon>0$ it is enough to show that there are infinitely many $n\in\e_2$ with 
$$\|n\alpha\|\ll n^{-\tau}.$$
Let $\frac{c}{q}$ be a convergent in the continued fraction expansion of $\alpha$ with a sufficiently large denominator.  We therefore have 
$$|\alpha-\frac{c}{q}|\leq \frac{1}{q^2}.$$
If we let $x=q^{\frac{2}{1+\tau}}$, $z=\frac{x}{q}$ and $a=\overline c$ then any $n\in\ca$ satisfies 
$$an\equiv k\pmod q\text{ for some }k\in [0,z].$$
We therefore have 
$$\|\frac{an}{q}\|\leq\frac{z}{q}.$$
It follows that 
$$\|n\alpha\|\leq\|(\alpha-\frac{c}{q})n\|+\|\frac{an}{q}\|\ll n^{-\tau}.$$
Since there are infinitely many convergents to $\alpha$ it is thus sufficient to show that $\ca$ contains members of $\e_2$.  This follows from Theorem \ref{mainthm}.  

\subsection{Proof of Theorem \ref{palindrome}}

Recall that 
$$\p_3(b)=\{j(b^2+1)+kb:j\in (0,b)\cap\Z,k\in [0,b)\cap\Z\}.$$
We take $\tau=\frac{1}{3},q=b^2+1$, $z=b$, $x=b^3$ and $a=b$.  The set $\ca$ is then contained in $\p_3(b)$ so the result follows from Theorem \ref{mainthm}.

\addcontentsline{toc}{section}{References} 
\bibliographystyle{plain}
\bibliography{biblio} 

\begin{thebibliography}{1}

\bibitem{dfi}
W.~Duke, J.~Friedlander, and H.~Iwaniec.
\newblock Bilinear forms with {K}loosterman fractions.
\newblock {\em Invent. Math.}, 128(1):23--43, 1997.

\bibitem{har83}
G.~Harman.
\newblock On the distribution of {$\alpha p$} modulo one.
\newblock {\em J. London Math. Soc. (2)}, 27(1):9--18, 1983.

\bibitem{har96}
G.~Harman.
\newblock On the distribution of {$\alpha p$} modulo one. {II}.
\newblock {\em Proc. London Math. Soc. (3)}, 72(2):241--260, 1996.

\bibitem{rhbjia}
D.~R. Heath-Brown and C.~Jia.
\newblock The distribution of {$\alpha p$} modulo one.
\newblock {\em Proc. London Math. Soc. (3)}, 84(1):79--104, 2002.

\bibitem{mykloost}
A.~J. Irving.
\newblock Average bounds for kloosterman sums over primes.
\newblock arXiv:1301.6372.

\bibitem{mat}
K.~Matom{\"a}ki.
\newblock The distribution of {$\alpha p$} modulo one.
\newblock {\em Math. Proc. Cambridge Philos. Soc.}, 147(2):267--283, 2009.

\bibitem{vaughan}
R.~C. Vaughan.
\newblock On the distribution of {$\alpha p$} modulo {$1$}.
\newblock {\em Mathematika}, 24(2):135--141, 1977.

\bibitem{vinogradov}
I.~M. Vinogradov.
\newblock {\em The method of trigonometrical sums in the theory of numbers}.
\newblock Dover Publications Inc., Mineola, NY, 2004.
\newblock Translated from the Russian, revised and annotated by K. F. Roth and
  Anne Davenport, Reprint of the 1954 translation.

\end{thebibliography}

\bigskip
\bigskip

Mathematical Institute,

24--29, St. Giles',

Oxford

OX1 3LB

UK
\bigskip

{\tt irving@maths.ox.ac.uk}

\end{document}